\documentclass[11pt,reqno]{amsart}

\usepackage{amscd,amssymb,amsmath,amsthm}
\usepackage[arrow,matrix]{xy}
\usepackage{graphicx}
\usepackage{color}
\usepackage{cite}
\newtheorem{thm}{Theorem}
\newtheorem{defn}{Definition}

\newtheorem{pro}{Proposition}

\newtheorem{cor}{Corollary}

\newtheorem{ex}{Example}

\numberwithin{equation}{section} 

\newcommand{\B}{{\mathcal B}}

\newcommand{\bea}{\begin{eqnarray}}
\newcommand{\eea}{\end{eqnarray}}

\newcommand{\C}{\mathbb{C}}

\def\C{\mathcal{C}}

\begin{document}

\title[Evolution algebra of a ``chicken'' population]{On subalgebras of an evolution algebra of a ``chicken'' population}

\author{B.A. Omirov,  U. A. Rozikov}

 \address{B. \ A. \ Omirov and U.\ A.\ Rozikov\\ Institute of mathematics,
29, Do'rmon Yo'li str., 100125, Tashkent, Uzbekistan.}
\email {omirovb@mail.ru, \ \ rozikovu@yandex.ru}

\begin{abstract} We consider an evolution algebra which corresponds to a
bisexual population with a set of
females partitioned into finitely many different types
and the males having only one type. For such algebras in
terms of its structure constants we calculate right and plenary
periods of generator elements. Some results on subalgebras of EACP
and ideals on low-dimensional EACP are obtained.
\end{abstract}
\maketitle



\section{Introduction} \label{sec:intro}

In recent years the non-commutative and non-associative analogies of the classical constructions become interesting in the connection with their applications in many branches of mathematics, biology (population, genetics, etc.) and physics (string theory, quantum field theory, etc.).

An algebraic approach in Genetics consists of the study of various types of genetic algebras (like algebras of free, "self-reproduction" and bisexual populations, Bernstein algebras) \cite{LR}, \cite{ly}. Mendel exploited symbols that are quite algebraically suggestive to express his genetic laws. The evolution of a population comprises a determined change of state in the next generations as a result of reproduction and selection \cite{ly},\cite{m}.

The main problem for a given algebra of a sex linked population is to carefully examine how
the basic algebraic model must be
altered in order to compensate for this lack of symmetry in the genetic inheritance
system. In \cite{e3}  Etherington began the study of this kind of algebras with the
simplest possible case.

Recently in \cite{LR} an evolution algebra $\B$ is introduced identifying the coefficients
of inheritance of a bisexual population as the structure constants
of the algebra.  The basic properties of the algebra are studied.
Moreover a detailed analysis of a special case
of the evolution algebra (of bisexual population in which type ``1''
of females and males have preference) is given. Since the structural constants of the algebra $\B$ are given by two cubic matrices, the study of this algebra is {\bf quite} difficult. To avoid such difficulties one has to consider an algebra of bisexual population with a simplified form of matrices of structural constants. In \cite{LRc} a such simplified model of bisexual population is considered and basic properties of corresponding evolution algebra (called evolution algebras of a ``chicken" population (CEACP)) are studied. In \cite{RM} a notion of chain of EACP is introduced and
several examples (time homogenous, time non-homogenous, periodic, etc.) of such chains are given.

In this paper we calculate right and plenary periods for generator elements of EACP and establish that natural basis of any subalgebra of EACP (which is also a EACP) can be extended to a natural basis of whole algebra. Moreover, we describe one-dimensional subalgebras (in ordinary sense) of EACP. Finally, simplicity of low-dimensional EACP is investigated.

\section{Basic definitions}

Following \cite{LRc} we consider a set  $\{h_i, i=1,\dots,n\}$ (the set of "hen"s) and $r$ (a "rooster").

\begin{defn} \cite{LRc} Let $(\mathcal E, \cdot)$ be an algebra over a field $K$. If it admits a basis $\{h_1,\dots,h_n, r\}$, such that
\begin{equation}\label{4}
\begin{array}{ll}
h_ir = rh_i=\sum_{j=1}^na_{ij}h_j+ b_ir,  \\
h_ih_j =0,\ \ i,j=1,\dots,n; \ \ rr =0 \end{array} \end{equation}
then this algebra is called an evolution algebra of a "chicken" population (EACP).
We call the basis $\{h_1,\dots,h_n, r\}$ a natural basis.
\end{defn}

Thus an algebra EACP, $\mathcal E$, is defined by a rectangular $n\times (n+1)$-matrix
$$M=\left(\begin{array}{ccccc}
a_{11}&a_{12}&\dots&a_{1n}&b_1\\[2mm]
a_{21}&a_{22}&\dots&a_{2n}&b_2\\[2mm]
\vdots&\vdots&\dots&\vdots&\vdots\\[2mm]
a_{n1}&a_{n2}&\dots&a_{nn}&b_n
\end{array}\right),$$
which is called the matrix of structural constants of the algebra $\mathcal E$.

We write the matrix $M$ in the form $M=A\oplus {\bf b}$ where $A=(a_{ij})_{i,j=1,\dots,n}$ and ${\bf b}^T=(b_1,\dots,b_n)$.

Assume we have two rectangular $n\times (n+1)$-matrices $M=A\oplus {\bf b}$ and $H=B\oplus {\bf c}$. Then we define multiplication of such matrices by
\begin{equation}\label{MH}
MH=AB\oplus A{\bf c},  \ \ HM=BA\oplus B{\bf b}.
\end{equation}
We note that this multiplication agrees with usual multiplication of $(n+1)\times (n+1)$-matrices with zero $(n+1)$-th row.

Let $E$ be a commutative algebra, define {\it principal} power of $a\in E$ as
$$a^2=a\cdot a, \ \ a^3=a^2\cdot a, \ \ \dots, \ \ a^n=a^{n-1}\cdot a;$$
and {\it plenary} powers of $a$ as
$$a^{[1]}=
%a^{(2)}=
a\cdot a, \ \
a^{[n]}=a^{[n-1]}a^{[n-1]}
%a^{(2^n)}=a^{(2^{n-1})}\cdot a^{(2^{n-1})}
,\ \ n\geq 2.$$
Define right multiplication operator by
$$R_a(x)=xa.$$

Let $\mathcal E$ be an EACP with the generator set $\{h_1,h_2,\dots,h_n,r\}$.
We say $h_i$ (or $r$) occurs in $x\in  \mathcal E$ if the coefficient $\alpha_i$ (or $a$)
in $x=\sum_{i=1}^n\alpha_ih_i+ar$ is non-zero. Write $h_i\prec x$ ($r\prec x$).

\begin{defn} Let $h_j$ be a generator of an EACP, the right period $p_j$ of $h_j$ is defined by
$$p_j=\min\{m\in N: h_j\prec R^m_r(h_j)\}.$$ If $p_j=1$, we say $h_j$ is aperiodic; if the set is empty we define $p_j=\infty$.
\end{defn}
\begin{defn} Let $h_j$ be a generator of an EACP, the plenary period $q_j$ of $h_j$ is defined by
$$q_j=\min\{m\in N: h_j\prec (h_jr)^{[m]}\}.$$ If $q_j=1$, we say $h_j$ is aperiodic;
if the set is empty we define $q_j=\infty$.
\end{defn}
\section{Conditions of periodicity}

\begin{pro}\label{p1} For any $m\geq 1$ and for any $i=1,\dots,n$ the following identities hold
\begin{itemize}
\item[(i)] $R^m_r(h_i)=(A^m\mathbf{h})_i+(A^{m-1}{\bf b})_ir;$
\item[(ii)] $(h_ir)^{[m]}=\gamma_m\left[(A^{m+1}\mathbf{h})_i+(A^{m}{\bf b})_ir\right],$
where $\mathbf{h}=\{h_1,\dots,h_n\}$ and $\gamma_m$ satisfies the recurrent equation:
$$\gamma_{m+1}=2\gamma_m^2(A^m{\bf b})_i, \ \ \mbox{with} \ \ \gamma_1=2b_i.$$
\end{itemize}
\end{pro}
\begin{proof}
(i) Compute actions of $R_r$ to the set $\mathbf{h}$:
$$R_r(\mathbf{h})=\{R_r(h_1),\dots,R_r(h_n)\}=
\{h_1r,\dots,h_nr\}=\{(M\mathbf{h})_1,\dots,(M\mathbf{h})_n\}, $$
where
$$(M\mathbf{h})_i=\sum_{j=1}^na_{ij}h_j+b_ir=(A\mathbf{h})_i+b_ir,\ \ i=1,\dots,n.$$
Also we have
$$R^2_r(h_i)=R_r((M\mathbf{h})_i)=\sum_{s=1}^n\sum_{j=1}^na_{ij}a_{js}h_s+\sum_{j=1}^na_{ij}b_jr=(A^2\mathbf{h})_i+(A{\bf b})_ir.$$
Using induction {\bf by} $m$ we get
$$R^m_r(\mathbf{h})=\{(M^m\mathbf{h})_1,\dots,(M^m\mathbf{h})_n\}, $$
where
$$(M^m\mathbf{h})_i=R^m_r(h_i)=(A^m\mathbf{h})_i+(A^{m-1}{\bf b})_ir,\ \ i=1,\dots,n.$$

(ii) Use induction {\bf by} $m\geq 1$. For $m=1$ we have
$$(h_ir)^{[1]}=\left(\sum_{j=1}^na_{ij}h_j+b_ir \right)^2=2b_i\left[(A^{2}\mathbf{h})_i+(A{\bf b})_ir\right].$$
Assume now that the formula (ii) is true for $m$ and prove it for $m+1$:
\begin{equation}\label{m}
(h_ir)^{[m+1]}=\left(\gamma_m\left[(A^{m+1}\mathbf{h})_i+(A^{m}{\bf b})_ir\right]\right)^2=2\gamma_m^2(A^m{\bf b})_i((A^{m+1}\mathbf{h})_ir).
\end{equation}
 Let $A^m=(a_{ij}^{(m)})_{ij=1,\dots,n}$. Then from (\ref{m}) we get
 \begin{equation}\label{m1}
(h_ir)^{[m+1]}=2\gamma_m^2(A^m{\bf b})_i(\sum_{j=1}^na_{ij}^{(m+1)}h_jr)=\gamma_{m+1}\left[(A^{m+2}\mathbf{h})_i+(A^{m+1}{\bf b})_ir\right].
\end{equation}
\end{proof}

As a corollary of this proposition we have
\begin{pro}\label{p2} 1)  The right period of $h_i$ is
$$p_i=\min\{m\in N: a^{(m)}_{ii}\ne 0\}.$$

2) If $b_i=0$ then $q_i=\infty$, otherwise
the plenary period of $h_i$ is
$$q_i=\min\{m\in N: a^{(m+1)}_{ii}\prod_{j=0}^{m-1}(A^j{\bf b})_i\ne 0\},$$
where $A^0=id$.
\end{pro}
\begin{proof} 1) This simply follows from the part (i) of Proposition \ref{p1}.

2) Using part (ii) of Proposition \ref{p1} we get
$$(h_ir)^{[m]}=2^{2^{m}-1}\prod_{j=0}^{m-1}(A^j{\bf b})_i^{2^{m-j-1}}\left[(A^{m+1}\mathbf{h})_i+(A^{m}{\bf b})_ir\right].$$
Thus the coefficient of $h_i$ is
$$2^{2^{m}-1}\prod_{j=0}^{m-1}(A^j{\bf b})_i^{2^{m-j-1}}a^{(m+1)}_{ii}.$$
This completes the proof.
\end{proof}

The following proposition reduces an EACP to a simple one.

\begin{pro}\label{p3}\cite{DOR}
Let $\mathcal{C}$ be an EACP, then there exists
a basis $\{h_1, h_2, \dots, h_n, r\}$ such that $\mathcal{C}$
on this basis is represented by the table of multiplication as
follows
$$h_1r=\sum_{j=1}^{n}a_{1j}h_j+\delta r, \ \delta\in\{0, 1\},\qquad
h_ir=\sum_{j=1}^{n}a_{ij}h_j, \quad 2\leq i \leq
n.$$
\end{pro}
Using this proposition by Proposition \ref{p2} we get
\begin{cor}\label{c1} For EACP mentioned in Proposition \ref{p3} the following hold

a) If $\delta=0$ then $q_i=1$ or $\infty$,

b) If $\delta=1$ then the plenary period of $h_i$ is
$$q_i\in\left\{\begin{array}{ll}
\{1,2,\infty\} \ \ \mbox{if} \ \ i=1\\[3mm]
\{1,\infty\}, \ \ \mbox{if} \ \ i\ne 1.
\end{array}\right.
$$
\end{cor}
\begin{proof} a) If $h_i$ is present in $h_ir$ then $q_i=1$, otherwise since $(h_ir)^{[m]}=0$ for all $m\geq 2$ we get $q_i=\infty$.

b) Case $i=1$. If $h_1$ is present in $h_1r$ then $q_1=1$, otherwise consider $(h_1r)^{[2]}$ if this contains $h_1$ then $q_1=2$, if $(h_1r)^{[2]}$ does not contain $h_1$ then since $(h_1r)^{[m]}=0$ for all $m\geq 3$ we get $q_i=\infty$.

Case $i\ne 1$ is similar to part a).
\end{proof}

\section{Subalgebras of an EACP}

By definition of an EACP we know that this algebra depends on a natural basis $\{h_1, h_2, \dots, h_n, r\}.$
\begin{defn}\label{subalgebra}\cite{LRc} \begin{itemize}
\item[1)] Let $\C$ be an EACP and $\C_1$ be a subspace of $\C$.
If $\C_1$ has a natural basis $\{h_1',h_2',\dots,h_m',r'\}$ with multiplication table like (\ref{4}), then we call $\C_1$ an evolution subalgebra of a CP.

\item[2)] Let $I\subset \C$ be an evolution subalgebra of a CP.
If $\C I\subseteq I$, we call $I$ an evolution ideal of a CP.

\item[3)] Let $\C$ and $\mathcal D$ be EACPs, we say a linear homomorphism $f$
from $\C$ to $\mathcal D$ is an evolution homomorphism, if $f$ is an algebraic map and for
a natural basis $\{h_1,\dots, h_n, r\}$ of $\C$, $\{f(h_1),\dots, f(h_n), f(r)\}$ spans an evolution subalgebra of a CP
in $\mathcal D$. Furthermore, if an evolution homomorphism is one to one and onto, it
is an evolution isomorphism.

\item[4)] An EACP $\C$ is called simple if it has no proper evolution ideals.

\item[5)] $\C$ is called irreducible if it has no proper subalgebras.
\end{itemize}
\end{defn}

In fact, for linear subspace $\C_1$ of an EACP $\C$ we can consider three type of subalgebras:

(i) $\C_1$ is a subalgebra in ordinary sense;

(ii) $\C_1$ is subalgebra and there exists a natural basis of $\C_1;$

(iii) $\C_1$ is subalgebra and there exist a natural basis of $\C_1$ which can be extended to a natural basis of $\C.$

Note that Definition \ref{subalgebra} agrees with the second type of subalgebra.

The following proposition gives equivalence of (ii) and (iii).

\begin{pro} Definitions (ii) and (iii) are equivalent.
\end{pro}
\begin{proof} Part $(iii)\Rightarrow (ii)$ is straightforward. We shall prove $(ii)\Rightarrow (iii)$.
Let $\C_1=\{f_1,f_2,\dots,f_m,r'\}$ be a subalgebra of $\C=\{h_1,\dots,h_n,r\}$ in sense (ii).
We shall show that the natural basis of $\C_1$ can be extended to a natural basis of $\C$.
We have
\begin{equation}\label{f}
\begin{array}{ll}
f_i=\sum_{j=1}^n\alpha_{ij}h_j+\gamma_i r,  \ \ i=1,\dots,m,\\[2mm]
r'=\sum_{j=1}^n\beta_jh_j+\gamma r.
\end{array}
\end{equation}
{\it Case} $\gamma\ne 0$. Take the following change of the basis
$$ f_i'=f_i-{\gamma_i\over \gamma}r', \quad 1 \leq i \leq m, \quad r''=r'.$$
This new basis also is a natural basis, moreover the vectors $f'_i$ do not contain $r$ in their
decompositions. Thus vectors $\{f'_1,\dots,f'_m\}$ generate a subspace in the vector space
generated by $\{h_1,\dots,h_n\}$. Then using theorem about change of basis (see \cite{W}) we can replace
 $\{h_{i_1},\dots, h_{i_m}\}$ by $\{f'_1,\dots,f'_m\}$. Moreover $r$ can be replaced by $r'$.  Hence for $\gamma\ne 0$ we can extend the natural basis of $\C_1$ to the natural basis of $\C$.

{\it Case} $\gamma=0$ and $\gamma_i=0$ for all $i$. In this case all $f_i$ and $r'$ do not depend on $r$. So
we can again use theorem about change of basis and replace $\{h_{i_1}, \dots , h_{i_m}, h_{i_{m+1}}\}$ by $\{f_1,\dots,f_m,r'\}$.

{\it Case}  $\gamma=0$ and $\gamma_i\ne 0$ for some $i$. By change $X_i=r'; X_j=f_j, j\ne i; r''=f_i$ we reduce this case to the first case. This completes the proof.

\end{proof}

The following is an example of a subalgebra (as in (i))
of $\C$, which is not an evolution subalgebra of a CP (as in (ii)).

\begin{ex}\cite{LRc} Let $\C$ be an EACP over a field $\mathbb{K}$ with basis $\{h_1,h_2,h_3,r\}$ and multiplication
defined by $h_ir=h_i+r$, $i=1,2,3$. Take $u_1=h_1+r$, $u_2=h_2+r$. Then
$$(au_1+bu_2)(cu_1+du_2)=acu_1^2+(ad+bc)u_1u_2+bdu_2^2=(2ac+ad+bc)u_1+(2bd+ad+bc)u_2.$$
Hence, $F=\mathbb{K}u_1+\mathbb{K}u_2$ is a subalgebra of $\C$, but it is not an evolution subalgebra of a CP. Indeed,
assume $v_1, v_2$ be a basis of $F$. Then $v_1=au_1+bu_2$ and $v_2=cu_1+du_2$ for some $a,b,c,d\in \mathbb{K}$ such that
$D=ad-bc\ne 0$. We have $v_1^2=(2a^2+2ab)u_1+(2b^2+2ab)u_2$ and $v_2^2=(2c^2+2cd)u_1+(2d^2+2cd)u_2$. We must have $v_1^2=v_2^2=0$, i.e.
$$ a^2+ab=0, \ \ b^2+ab=0, \ \ c^2+cd=0, \ \ d^2+cd=0.$$

From this we get $a=-b$ and $c=-d$. Then $D=0$, a contradiction. If $a=0$ then $b=0$ (resp. $c=0$ then $d=0$), we reach the same contradiction. Hence $v_1^2\ne 0$ and $v_2^2\ne 0$, and consequently $F$ is not an evolution subalgebra of a CP.
\end{ex}

In sequel for a subalgebra we mean a subalgebra in the sense (iii).

\begin{pro}\label{pp} Let $\C$ be an EACP over $\mathbb{R}$ with basis $\{h_1,\dots,h_n,r\}$ and matrix of structural constants $M=A\oplus {\bf b}$. If
${\rm rank} A=n,$ then any subalgebra of $\C$ has the form $\{f_1,\dots,f_m, ar\},$ where $0\leq m\leq n$, $a\in \{0,1\}$ and
$$f_i=\sum_{j=1}^n\alpha_{ij}h_j, \ \ \alpha_{ij}\in \mathbb{R}, i=1,\dots,m.$$
\end{pro}
\begin{proof} Let ${\tilde \C}=\{\varphi_1,\dots,\varphi_m\}$ be a subalgebra of $\C$. Then we have
  $$\varphi_i=\sum_{k=1}^n\beta_{ik}h_k+\beta_ir, \ \ i=1,\dots,m.$$
 Since $\varphi_i^2=0$, then we have
\begin{equation}
  2\beta_i\sum_{k=1}^n\beta_{ik}h_kr=2\beta_i\left(\sum_{k=1}^n\sum_{s=1}^n\beta_{ik}a_{ks}h_s+\sum_{k=1}^n\beta_{ik}b_kr\right)=0.
\end{equation}
Hence $\beta_i=0$ or
\begin{equation}\label{s}
\sum_{k=1}^n\beta_{ik}a_{ks}=0 \ \mbox{for \ any \ } \  s \ \ \mbox{and}  \ \ \sum_{k=1}^n\beta_{ik}b_k=0.
\end{equation}
Since ${\rm rank} A=n$ from (\ref{s}) we get $\beta_{ik}=0$ for all $k$.
Hence $\varphi_i$ is equal to $\beta_i r$ or to $\sum_{k=1}^n\beta_{ik}h_k$. This completes the proof.
\end{proof}

\begin{pro}\label{pp} Let $\C$ be an EACP  with matrix of structural constants $M=A\oplus {\bf b}$.
Then $\mathcal{X}=\langle x\rangle$, where $0\ne x=y+\beta r=\sum_{i=1}^n\alpha_ih_i+\beta r$ generates an one-dimensional subalgebra if one of the following conditions is satisfied
\begin{itemize}

\item[a.] $\beta=0$ or $Ay=0$, ${\bf b}y=0$.

 \item[b.] $\beta\ne 0$, ${\bf b}y=1$ and $y$ is an eigenvector of $A$ with eigenvalue $1/\beta$.
\end{itemize}
\end{pro}
\begin{proof} An arbitrary $x=\sum_{i=1}^n\alpha_ih_i+\beta r$ generates a subalgebra iff $x^2=cx$ for some $c$.
Here one can consider only the case $c=0$ and $c=1$.  Thus $x$ generates a subalgebra iff it is an absolute nilpotent or idempotent of $\C$. Now the proof follows from Propositions 3.4 and 3.5 of \cite{LRc}.
\end{proof}

\begin{pro}\label{pi} Let $\C$ be an EACP as in Proposition \ref{p3}, $\delta=1$ and with matrix of structural constants $M=A\oplus {\bf b}$.
Then $\mathcal{X}=\langle x\rangle$, where $x=\sum_{i=1}^n\alpha_ih_i+\beta r$ generates an one-dimensional ideal iff
one of the following conditions is satisfied
\begin{itemize}

\item[a.] $\beta=\alpha_1=\sum_{i=2}^na_{i1}\alpha_i=0$ and $x$ (with $\alpha_1=0$) is an eigenvector of $A_1$ with a real eigenvalue, where $A_1=\left(a_{ij}\right)_{i,j=2,\dots,n}$ is the minor of the matrix $A$.

\item[b.] $\beta=1$ and $\alpha_j=a_{1j}$ and $a_{kj}=0$, for all $k=2, \dots, n$, $j=1, \dots, n$.
\end{itemize}
\end{pro}
\begin{proof} Take an arbitrary element $y=\sum_{i=1}^n\gamma_ih_i+\nu r\in \C$ we should have $xy\in \mathcal{X}$, i.e. there exists $c$ such that $xy=cx$. The last equality is equivalent to
\begin{equation}\label{o}
\left\{\begin{array}{ll}
\sum_{i=1}^n(\nu\alpha_i+\beta\gamma_i)a_{ij}=c\alpha_j; \ \ j=1,2,\dots,n\\[3mm]
\nu \alpha_1+\beta \gamma_1=c\beta.
\end{array}
\right.\end{equation}

a. For case $\beta=0$ if $\nu=0$ then in (\ref{o}) one can take $c=0$. If $\nu\ne 0$ then $\alpha_1=0$ and
 $$\begin{array}{ll}
\nu\sum_{i=2}^n\alpha_ia_{ij}=c\alpha_j; \ \ j=2,\dots,n\\[3mm]
\sum_{i=2}^n\alpha_ia_{i1}=0.
\end{array}
$$
This completes the proof of a.

b. In the case $\beta\ne 0$ one can take $\beta=1$.
For $y=h_k$, $k=2,\dots,n$ from (\ref{o}) for some $c=c_k$ we get
the system $a_{kj}=c_k\alpha_j$, $j=1,\dots,n$ and $c_k=0$. This implies $a_{kj}=0$ for all $k=2,\dots,n$ and $j=1,\dots,n$. In case $y=h_1$ we get the system $a_{1j}=c_1\alpha_j$, $j=1,\dots,n$ and $c_1=1$. Hence $a_{1j}=\alpha_j$. Taking into account
the above obtained results, for $y=r$ we get $\alpha_1a_{1j}=c\alpha_j$ and $\alpha_1=c$. Thus we proved that if $A$ has the following form
$$A=\left(\begin{array}{cccccc}
\alpha_1&\alpha_2&\dots&\alpha_n\\[2mm]
0&0&\dots&0\\[2mm]
\vdots&\vdots&\vdots&\vdots\\[2mm]
0&0&\dots&0
\end{array}
\right)$$
then there are $c_k$ and $c$ such that $xy=c_kx$ if $y=h_k$ and $xy=cx$ if $y=r$. Using this result for an arbitrary $y=\sum_{i=1}^n\gamma_ih_i+\nu r\in \C$ we obtain
$$xy=\sum_{i=1}^n\gamma_ixh_i+\nu xr=(\sum_{i=1}^n\gamma_ic_i+c)x=Cx.$$
Thus $\mathcal{X}=\langle x=\sum_{i=1}^n\alpha_ih_i+r\rangle$ is an ideal of the algebra $\C$ with matrix
$$M=\left(\begin{array}{cccccc}
\alpha_1&\alpha_2&\dots&\alpha_n&1\\[2mm]
0&0&\dots&0&0\\[2mm]
\vdots&\vdots&\vdots&\vdots&\vdots\\[2mm]
0&0&\dots&0&0
\end{array}
\right)$$

\end{proof}

\section{Simple three-dimensional complex EACPs}

In the following theorem the classification of three
dimensional EACP is presented.

\begin{thm}\label{tt2}\begin{itemize}

\item[1.] \cite{LRc} Any  2-dimensional, non-trivial EACP $\C$ is isomorphic to
one of the following pairwise non isomorphic algebras:
\begin{itemize}
\item[$\C_1$:]\ \ $rh =hr=h$,  \ \ $h^2=r^2 = 0$,
\item[$\C_2$:]\ \ $rh=hr={1\over 2}(h+r)$, \ \ $h^2=r^2 = 0$.
\end{itemize}
\item[2.]\cite{DOR}
 An arbitrary three dimensional complex EACP
$\mathcal{C}$ is isomorphic to one of the following pairwise
non-isomorphic algebras

If dim $\mathcal{C}^2=1$ then
$$\begin{array}{ll} \mathcal{C}_1: & h_1r=\frac{1}{2}r\\[1mm]
\mathcal{C}_2: & h_1r=\frac{1}{2}h_2;\\[1mm]
\mathcal{C}_3: & h_1r=\frac{1}{2}h_1+\frac{1}{2}r.
\end{array}
$$

If dim $\mathcal{C}^2=2$ then
$$\begin{array}{llll}
\mathcal{C}_4:& h_1r=\frac{1}{2}(h_1+h_2),&
h_2r=\frac{1}{2}h_2;\\[1mm]
\mathcal{C}_5(\beta): & h_1r=\frac{1}{2}h_1,&
h_2r=\frac{\beta}{2}h_2,& \beta\neq0;\\[1mm]
\mathcal{C}_6(\alpha, \beta): & h_1r=\frac{1}{2}(\alpha h_1+\beta
h_2+r),& h_2r=\frac{1}{2}h_1;\\[1mm]
\mathcal{C}_7(\alpha): & h_1r=\frac{1}{2}(\alpha h_1+r), &
h_2r=\frac{1}{2}h_2;\\[1mm]
\mathcal{C}_8: & h_1r=\frac{1}{2}(h_1+h_2+r), &
h_2r=\frac{1}{2}h_2.\end{array}$$
where one of non-zero parameter
$\alpha, \beta$ in the algebra $\mathcal{C}_6(\alpha, \beta)$ can
be assumed to be equal to 1.
\end{itemize}
\end{thm}

The following theorem describes simple and not simple EACP listed in Theorem \ref{tt2}.
\begin{thm}\begin{itemize}
\item[a.] The two-dimensional algebra $\C_1$ and $\C_2$ are not simple.

\item[b.] The three-dimensional algebra $\C_i$ is not simple for $i=1,2,3,4,5,7,8$ and $i=6$ for $\beta=0$. Moreover, $\C_6(\alpha,\beta)$  is simple for $\beta\ne 0$.
\end{itemize}
\end{thm}
\begin{proof} a. It is easy to see that $\langle h\rangle\lhd \C_1$ and $\langle h+r\rangle\lhd \C_2$.

b. Consider some possible subalgebras (in sense (iii)) of $\C=\{h_1,h_2,r\}$:
$$D_1=\{h_1\}, \ \ D_2=\{h_1, h_2\}, \ \ D_3=\{h_1,r\},$$
$$D_4=\{h_2\}, \ \ D_5=\{h_2, r\}, \ \ D_6=\{r\}.$$
It is easy to check that
$$D_j=\left\{\begin{array}{llllllll}
is\ \ ideal\ \ for\ \ \C_1 \ \ if\ \ j=3,4,5,6 \ \ and \ \ is \ \ not\ \ ideal\ \ if\ \ j=1,2;\\[2mm]
is\ \ ideal\ \ for\ \ \C_2\ \ if\ \ j=2,4,5 \ \ and \ \ is\ \ not\ \ ideal \ \ if\ \ j=1,6;\\[2mm]
is\ \ ideal\ \ for\ \ \C_3\ \ if\ \ j=3,4 \ \ and \ \ is\ \ not\ \ ideal \ \ if\ \ j=1,2,5,6;\\[2mm]
is\ \ ideal\ \ for\ \ \C_4\ \ if\ \ j=2,4,5 \ \ and \ \ is\ \ not\ \ ideal \ \ if\ \ j=1,6;\\[2mm]
is\ \ ideal\ \ for\ \ \C_5\ \ if\ \ j=1,2,4, \ \ and \ \ is\ \ not\ \ ideal \ \ if\ \ j=3,5,6;\\[2mm]
is\ \ not\ \ ideal \ \ for\ \ \C_6\ \ if \ \ j=1,2,4,6;\\[2mm]
is\ \ ideal\ \ for\ \ \C_7\ \ if\ \ j=4 \ \ and \ \ is\ \ not\ \ ideal \ \ if\ \ j=1,2,3,5,6;\\[2mm]
is\ \ ideal\ \ for\ \ \C_8\ \ if\ \ j=4 \ \ and \ \ is\ \ not\ \ ideal \ \ for\ \ j=1,2,5,6.\\[2mm]
\end{array}\right.$$

Now consider $\C_6$:

{\it Case} $\beta=0$.  In this case $D_3$ will be an ideal, i.e. $\C_6(\alpha,0)$ is not simple.

{\it Case} $\beta\ne 0$. This $\beta$ can be reduced to $\beta=1$. We have ${\rm rank} A=2$. So we can use Proposition \ref{pp}:  consider a general subalgebra $\tilde\C_6=\{ah_1+bh_2, \delta r\}$. For $\delta=0$ it is easy to see that $\tilde\C_6\C_6\not\subset \tilde\C_6$. If $\delta=1$ then
$$\tilde\C_6\C_6=\{(a\alpha+b)h_1+ah_2+ar, \alpha h_1+h_2+r, h_1\}.$$
Simple calculations show that $ (a\alpha+b)h_1+ah_2+ar\in \tilde\C_6$ iff $b=-{a\over 2}\cdot(\alpha\mp\sqrt{\alpha^2+4})$.
For this value of $b$ one gets  $\alpha h_1+h_2+r\in \tilde \C_6$ iff $\alpha\sqrt{\alpha^2+4}=\alpha^2+2$. But the last equation has not solution. Hence $\C_6(\alpha,\beta)$ is simple for any $\beta\ne 0$.
\end{proof}

\section*{ Acknowledgements}

The work supported by the Grant No.0251/GF3 of Education and Science Ministry of Republic
of Kazakhstan. U.Rozikov thanks Aix-Marseille University Institute for Advanced Study IM\'eRA
(Marseille, France) for support by a residency scheme.

{}
\end{document}